\documentclass[12pt]{article}

\usepackage{amssymb,amsmath,amsthm}
\usepackage{fullpage}
\usepackage{url}

\author{J.-P. Allouche
\\ CNRS, IMJ-PRG \\
Universit\'e P.\ et M.\ Curie \\
Case 247, 4 Place Jussieu \\
F-75252 Paris Cedex 05  \\
France \\
{\tt jean-paul.allouche@imj-prg.fr} \\
\and
S. Riasat and J. Shallit \\
School of Computer Science \\
University of Waterloo \\
Waterloo, ON  N2L 3G1 \\
Canada\\
{\tt sriasat@uwaterloo.ca} \\
{\tt shallit@cs.uwaterloo.ca} \\
}

\title{More Infinite Products: \\
Thue-Morse and the Gamma function}

\date{ }

\theoremstyle{plain}
\newtheorem{theorem}{Theorem}[section]
\newtheorem{lemma}[theorem]{Lemma}
\newtheorem{corollary}[theorem]{Corollary}

\theoremstyle{definition}

\newtheorem{remark}[theorem]{Remark}

\begin{document}

\maketitle

\begin{abstract}
Letting $(t_n)$ denote the Thue-Morse sequence with values $0, 1$,
we note that the Woods-Robbins product 
$$
\prod_{n \geq 0} \left(\frac{2n+1}{2n+2}\right)^{(-1)^{t_n}} = 2^{-1/2}
$$
involves a rational function in $n$
and the $\pm 1$ Thue-Morse sequence $((-1)^{t_n})_{n \geq 0}$.  
The purpose of this paper is twofold.  On the one hand,
we try to find other rational functions  for 
which similar infinite products involving the
$\pm 1$ Thue-Morse sequence have an expression in terms of known constants.
On the other hand,
we also try to find (possibly different) rational functions $R$ for which the 
infinite product $\prod R(n)^{t_n}$ also has an expression in terms of known constants.
\end{abstract}

 \section{Introduction}
Several infinite products involving the sum of binary digits of the integers were inspired by
the discovery of the Woods and Robbins infinite product (see \cite{Woods, Robbins}). 
More precisely, letting $t_n$ denote the sum, modulo $2$, of the binary digits of the integer 
$n$, the sequence $(t_n)_{n \geq 0} = 0 \ 1 \ 1 \  0 \  1 \  0 \  0 \ 1 \  1 \  0 \  0 \  1 \ldots$ is called the
Thue-Morse sequence with values $0$ and $1$ (see, e.g., \cite{ubiquitous} and the references
therein). The Woods-Robbins product identity is
\begin{equation}\label{WR}
\prod_{n \geq 0} \left(\frac{2n+1}{2n+2}\right)^{(-1)^{t_n}} = \frac{1}{\sqrt{2}} \cdot
\end{equation}
Several  infinite products inspired by (\ref{WR}) were discovered later (see, e.g., 
\cite{ACMFS, AS-infinite-products, All, All-Sondow, Hu}). They all involve,
as exponents, sequences 
of the form $(-1)^{u_{w,b}(n)}$ where $u_{w,b}(n)$ is the number, reduced modulo $2$, of 
occurrences of the word (the block) $w$ in the $b$-ary expansion of the integer $n$. But none of 
these products are in terms of
0-1-sequences $(u_{w,b}(n))_{n \geq 0}$ alone. In particular, none of them 
are in terms of
the binary sequence $(t_n)_{n \geq 0} = (u_{1,2}(n))_{n \geq 0}$ given above.
Furthermore, there has been
no attempt up to now to find explicitly-given
large classes of rational functions $R$ for which the infinite product
$\prod R(n)^{(-1)^{t_n}}$ has an expression in terms of known constants.

\bigskip

The purpose of this paper is thus twofold. First, to find other infinite products of the form
$\prod R(n)^{(-1)^{t_n}}$ admitting an expression in terms of known
constants. Second, to find infinite products of 
the form $\prod R(n)^{t_n}$ also having an expression in terms of
known constants. Two examples that we find are
\begin{equation*}
\prod_{n \geq 0} \left(\frac{4n+1}{4n+3}\right)^{(-1)^{t_n}} = \frac{1}{2}
\cdot
\end{equation*}
\begin{equation*}
\prod_{n \geq 0} \left(\frac{(4n+1)(4n+4)}{(4n+2)(4n+3)}\right)^{t_n} = \frac{\pi^{3/4} \sqrt{2}}{\Gamma(1/4)}  \cdot
\end{equation*}

\section{Products of the form $\prod R(n)^{(-1)^{t_n}}$}

We start with a lemma about the convergence of infinite products involving 
the sequence $((-1)^{t_n}))$.

\begin{lemma}\label{num-den}
Let $t_n$ be the sum, reduced modulo $2$, of the binary digits of the integer $n$. Let 
$R \in {\mathbb C}(X)$ be a rational function such that the values $R(n)$ are defined for $n \geq 1$. 
Then the infinite product \, $\prod_n R(n)^{(-1)^{t_n}}$ converges if and only if the numerator and 
the denominator of $R$ have same degree and same leading coefficient.
\end{lemma}

\begin{proof}
If the infinite product converges, then $R(n)$ must tend to $1$ when $n$ tends to infinity.
Thus the numerator and  the denominator of $R$ have the same degree and the same leading
coefficient.

\smallskip

Now suppose that the numerator and the denominator of $R$ have the same leading coefficient 
and the same degree. Decomposing them into factors of degree $1$, it suffices, for proving that the 
infinite product converges, to show that infinite products of the form 
$\prod_{n \geq 1} \left(\frac{n+b}{n+c}\right)^{(-1)^{t_n}}$ converge for complex numbers $b$ and
$c$ such that $n+b$ and $n+c$ do not vanish for any $n \geq 1$. Since the general factor of such
a product tends to $1$, this is equivalent, grouping the factors pairwise,
to proving that the product
$$\prod_{n \geq 1} \left[\left(\frac{2n+b}{2n+c}\right)^{(-1)^{t_{2n}}} 
\left(\frac{2n+1+b}{2n+1+c}\right)^{(-1)^{t_{2n+1}}}\right]$$ converges. 
Since $(-1)^{t_{2n}} = (-1)^{t_n}$ and $(-1)^{t_{2n+1}} = - (-1)^{t_n}$ we only need to prove that
the infinite product 
$$\prod_{n \geq 1} \left(\frac{(2n+b)(2n+1+c)}{(2n+c)(2n+1+b)}\right)^{(-1)^{t_n}}$$
converges. Taking the (principal determination of the) logarithm, we see that
$$\log\left(\frac{(2n+b)(2n+1+c)}{(2n+c)(2n+1+b)}\right) = {\mathcal O}(1/n^2),$$ which gives the
convergence result. 
\end{proof}

\bigskip

In order to study the infinite product $\prod_{n \geq 1} R(n)^{(-1)^{t_n}}$, it suffices, using 
Lemma~\ref{num-den} above, to study products of the form $\prod_n \left(\frac{n+a}{n+b}
\right)^{(-1)^{t_n}}$ where $a$ and $b$ belong to ${\mathbb C} \setminus \{-1, -2, -3, \ldots\}$.

\begin{theorem}\label{f-g}
Define
$$
f(a,b) := \prod_{n \geq 1} \left(\frac{n+a}{n+b}\right)^{(-1)^{t_n}} \ \ \mbox{\rm and} \ \ \ \
g(x) := \frac{f(\frac{x}{2}, \frac{x+1}{2})}{x+1}
$$
for $a$, $b$, $x$ complex numbers that are not negative integers.
Then 
$$
f(a,b) = \frac{g(a)}{g(b)} \cdot
$$
Furthermore, $g$ satisfies the functional equation
$$
(1+x) g(x) = \frac{g(\tfrac{x}{2})}{g(\tfrac{x+1}{2})} \ \ \ \ \
\forall x \in {\mathbb C} \setminus \{-1, -2, -3, \ldots \}.
$$
In particular we have $g(1/2) = 1$ and $g(1) = \sqrt{2}/2$.
\end{theorem}

\begin{proof}
Recall that $(-1)^{t_{2n}}$ and $(-1)^{t_{2n+1}} = - (-1)^{t_n}$. Hence
\begin{align*} 
f(a,b) &= {\prod_{n \geq 1} \left(\frac{n+a}{n+b}\right)^{(-1)^{t_n}}
= \prod_{n \geq 1} \left(\frac{2n+a}{2n+b}\right)^{(-1)^{t_{2n}}}
       \prod_{n \geq 0} \left(\frac{2n+1+a}{2n+1+b}\right)^{(-1)^{t_{2n+1}}}} \\
&= {\prod_{n \geq 1} \left(\frac{(2n+a)(2n+1+b)}{(2n+b)(2n+1+a)}\right)^{(-1)^{t_n}}
\left(\frac{1+b}{1+a}\right)} \\
&= {\prod_{n \geq 1}
\left( \frac{(n + \frac{a}{2})(n + \frac{1+b}{2})}{(n + \frac{a+1}{2})(n + \frac{b}{2})} \right)^{(-1)^{t_n}}\left(\frac{1+b}{1+a}\right)} \\[2pt]
&= {\frac{f(\frac{a}{2}, \frac{a+1}{2})}{f(\frac{b}{2}, \frac{b+1}{2})} \left(\frac{1+b}{1+a}\right)} 
= \frac{g(a)}{g(b)}  \cdot
\end{align*} 
Now taking $a = \frac{x}{2}$ and $b = \frac{x + 1}{2}$ in the equality $\frac{g(a)}{g(b)} = f(a,b)$ 
yields
$$ 
\frac{g(\frac{x}{2})}{g(\frac{x+1}{2})} = f(\tfrac{x}{2}, \tfrac{x+1}{2}) = (x+1) g(x),
$$ 
which is the announced functional equation.

Finally putting $x=0$ in this functional equation, and noting that $g(0) \neq 0$, yields
$g(\frac{1}{2}) = 1$, while putting $x=1$ gives $g(1)^2 = \frac{1}{2}$, hence 
$g(1) = \frac{1}{\sqrt{2}}$, since $g(1) > 0$.  
\end{proof}

\bigskip

This theorem implies many identities, including the original one
of Woods-Robbins (W.-R.).

\begin{corollary}\label{general-examples}
Let $a$ and $b$ belong to ${\mathbb C} \setminus \{-1, -2, -3, \ldots\}$. Then the 
following equalities hold.
$$
\begin{array}{lll}
(i) \ \ \ \ \ &\displaystyle\prod_{n \geq 1}\left(\frac{(n+a)(2n+a+1)(2n+b)}{(2n+a)(n+b)(2n+b+1)}
\right)^{(-1)^{t_n}}
&= \displaystyle\frac{b+1}{a+1} \\
(ii) \ \ \ \ \ &\displaystyle\prod_{n \geq 1} \left(\frac{(n+a)(2n+a+1)^2}{(2n+a)(2n+a+2)(n+a+1)}
\right)^{(-1)^{t_n}}
&= \displaystyle\frac{a+2}{a+1} \\
(iii) \ \ \ \ \ &\displaystyle\prod_{n \geq 1} \left(\frac{(2n+2a)(2n+a+1)}{(2n+a)(2n+1)}\right)^{(-1)^{t_n}}
&= \displaystyle\frac{1}{a+1} \cdot
\end{array}
$$
and, for $a \in {\mathbb C} \setminus \{0, -1, -2, -3, \ldots \} \cup \{-1/2, -3/2, -5/2, \ldots \})$,
$$ 
(iv) \ \ \ \ \ \prod_{n \geq 1}\left(\frac{(2n+a+1)(2n+2a-1)}{(2n+a)(2n+4a-2)}
\right)^{(-1)^{t_n}}
= \frac{2a}{a+1} \cdot \ \ \ \ \ \ \ \ \ \ 
$$
\end{corollary}

\begin{proof}
(i) is proved by writing its left side, say $A$, in terms of values of $f$ and applying 
Theorem~\ref{f-g}:
$$
A= f(a,\tfrac{a}{2})f(\tfrac{{a+1}}{2},b)f(\tfrac{b}{2}, \tfrac{b+1}{2})
= \frac{g(a)}{g(\frac{a}{2})} \frac{g(\frac{a+1}{2})}{g(b)} \frac{g(\frac{b}{2})}{g(\frac{b+1}{2})}
=  \frac{b+1}{a+1} \cdot
$$

(ii) is obtained from (i) by taking $b = a+1$.

(iii) is obtained from (i) by taking $b=0$.

(iv) is obtained from (i) by taking $b = 2a-1$. 
\end{proof}

\bigskip

We give examples with particular values of the parameters in the next corollary.

\begin{corollary}\label{examples}
We have the following equalities.
$$
\begin{array}{lll}
&(a) \ \ \ 
\displaystyle\prod_{n \geq 0} \left(\frac{2n+1}{2n+2}\right)^{(-1)^{t_n}} 
= \displaystyle\frac{\sqrt{2}}{2} \ \ 
\mbox{\rm (W.-R.)} 
\ \ \ \ \  &(b) \ \ \ \displaystyle\prod_{n \geq 0} \left(\frac{4n+1}{4n+3}\right)^{(-1)^{t_n}} 
= \displaystyle\frac{1}{2} 
\\
&(c) \ \ \ 
\displaystyle\prod_{n \geq 1} \left(\frac{(2n-1)(4n+1)}{(2n+1)(4n-1)}\right)^{(-1)^{t_n}} = 2 
\ \ \ \ \ 
&(d) \ \ \
\displaystyle\prod_{n \geq 0}\left(\frac{(n+1)(2n+1)}{(n+2)(2n+3)}\right)^{(-1)^{t_n}} = \frac{1}{2}
\\
&(e) \ \ \
\displaystyle\prod_{n \geq 0} \left(\frac{(2n+2)(4n+3)}{(2n+3)(4n+5)}\right)^{(-1)^{t_n}} =
\frac{\sqrt{2}}{2}
\ \ \ \ \ 
&(f) \ \ \
\displaystyle\prod_{n \geq 0} \left(\frac{(n+1)(4n+5)}{(n+2)(4n+3)}\right)^{(-1)^{t_n}} = 1 
\\
&(g) \ \ \
\displaystyle\prod_{n \geq 0} \left(\frac{(n+1)(2n+2)}{(n+2)(2n+3)}\right)^{(-1)^{t_n}} =
\frac{\sqrt{2}}{2} 
\ \ \ \ \ 
&(h) \ \ \
\displaystyle\prod_{n \geq 0} \left(\frac{(n+1)(4n+5)}{(n+2)(4n+1)}\right)^{(-1)^{t_n}} = 2
\\
&(i) \ \ \
\displaystyle\prod_{n \geq 0} \left(\frac{(2n+2)(4n+1)}{(2n+3)(4n+5)}\right)^{(-1)^{t_n}} = 
\frac{\sqrt{2}}{4}
\ \ \ \ \ 
&(j) \ \ \
\displaystyle\prod_{n \geq 0} \left(\frac{(2n+1)(4n+1)}{(2n+3)(4n+5)}\right)^{(-1)^{t_n}} = 
\frac{1}{4}
\\
&(k) \ \ \ 
\displaystyle\prod_{n \geq 0} \left(\frac{(4n+1)(8n+7)}{(4n+2)(8n+3)}\right)^{(-1)^{t_n}} = 1 
\ \ \ \ \
&(l) \ \ \ 
\displaystyle\prod_{n \geq 0} \left(\frac{(8n+1)(8n+7)}{(8n+3)(8n+5)}\right)^{(-1)^{t_n}} = 
\frac{1}{2}  \cdot
\\
\end{array}
$$
\end{corollary}

\begin{proof}
Corollary~\ref{general-examples}~(ii) with $a = 0$ yields
$$\prod_{n \geq 1}\left(\frac{(2n+1)^2}{(2n+2)^2}\right)^{(-1)^{t_n}} = 2.$$
Taking the square root, and multiplying by the value of $\frac{2n+1}{2n+2}$ for $n=0$, we obtain
the Woods-Robbins identity (a).

\medskip

\noindent
Corollary~\ref{general-examples}~(iii) with $a = \frac{1}{2}$ gives
$$\prod_{n \geq 1}\left(\frac{2n+\frac{3}{2}}{2n+\frac{1}{2}}\right)^{(-1)^{t_n}}
= \frac{2}{3}  \cdot $$
This implies (b) (note the different range of multiplication again).

\medskip

\noindent
Corollary~\ref{general-examples}~(iii) with $a = \frac{1}{2}$ gives
$$\prod_{n \geq 1}
\left(\frac{(2n-1)(2n + \frac{1}{2})}{(2n-\frac{1}{2})(2n+1)}\right)^{(-1)^{t_n}} = 2,$$
which implies (c).

\medskip

\noindent
Corollary~\ref{general-examples}~(i) with $a = 1$ and $b=2$ yields
$$\prod_{n \geq 1}\left(\frac{(n+1)(2n+2)^2}{(2n+1)(n+2)(2n+3)}\right)^{(-1)^{t_n}} =
\frac{3}{2} \cdot$$ 
We obtain (d) after multiplying by the factor corresponding to $n=0$, then by
the square of $\prod_{n \geq 0}\left(\frac{2n+1}{2n+2}\right)^{(-1)^{t_n}}$ (this square is equal to
$\frac{1}{2}$ from the identity of Woods and Robbins).

\medskip

\noindent
Corollary~\ref{general-examples}~(i) with $a=1$ and $b=\frac{3}{2}$ yields
$$\prod_{n \geq 1}
\left(\frac{(n+1)(2n+2)(2n+\frac{3}{2})}{(2n+1)(n+\frac{3}{2})(2n+\frac{5}{2})}\right)^{(-1)^{t_n}} 
= \frac{5}{4} \cdot$$ 
Equality~(e) is then obtained by multiplying by the factor corresponding to 
$n=0$ and then multiplying by $\prod_{n \geq 0}\left(\frac{2n+1}{2n+2}\right)^{(-1)^{t_n}}$,
(which again is equal to $\frac{\sqrt{2}}{2}$).

\medskip

\noindent
Corollary~\ref{general-examples}~(i) with $a=2$ and $b=\frac{3}{2}$ gives
$$\prod_{n \geq 1}
\left(\frac{(n+2)(2n+3)(2n+\frac{3}{2})}{(2n+2)(n+\frac{3}{2})(2n+\frac{5}{2})}\right)^{(-1)^{t_n}}
= \frac{5}{6} \cdot$$ 
We simplify by $(2n+3)$, multiply by the factor corresponding to $n=0$, and we
obtain (the inverse of) Equality (f).

\medskip

\noindent
Corollary~\ref{general-examples}(ii) with $a=1$ gives (g) with the usual manipulations (multiplying 
by the factor for $n=0$ and by $\prod_{n \geq 0}\left(\frac{2n+1}{2n+2}\right)^{(-1)^{t_n}}$).
Alternatively (g) can be obtained by multiplying (e) and (f). 

\medskip

\noindent
Equality~(h) is obtained by dividing (f) by (b). The inverse of Equality~(i) is obtained by dividing 
(h) by (g). Equality~(j) is obtained by multiplying (i) 
by $\prod_{n \geq 0}\left(\frac{2n+1}{2n+2}\right)^{(-1)^{t_n}}$,
(which is equal to $\frac{\sqrt{2}}{2}$).

\medskip

\noindent
Corollary~\ref{general-examples}~(iv) with $a = \frac{3}{4}$ yields
$$\prod_{n \geq 1}\left(\frac{(2n+\frac{7}{4})(2n+\frac{1}{2})}{(2n+\frac{3}{4})(2n+1)}\right)^{(-1)^{t_n}} = \frac{6}{7},$$ which implies (k).

\medskip

\noindent
Corollary~\ref{general-examples}~(i) with $a=\frac{3}{4}$ and $b=\frac{1}{4}$ 
gives (l) 
(multiply by the factor corresponding to $n=0$ and use (b)). 
\end{proof}

\begin{remark}

The proofs that we give, e.g., in Corollary~\ref{examples}, provide infinite products whose
values are rational: in the case of the Woods-Robbins infinite product 
$P = \prod_{n \geq 0} \left(\frac{2n+1}{2n+2}\right)^{(-1)^{t_n}}$ we actually obtain the value of
$P^2$ ($= 1/2$). We finally get $\frac{\sqrt{2}}{2}$ only because the product we first obtain involves
the square of a rational function.

\end{remark}

\section{More remarks on the function $g$}\label{FM}

As we have seen above, the function $g$ defined by $g(x) := \frac{f(\frac{x}{2}, \frac{x+1}{2})}{x+1}$
has the property that $f(a,b) = \frac{g(a)}{g(b)}$. It satisfies the functional equation
$\frac{g(\frac{x}{2})}{g(\frac{x+1}{2})} = (1+x)g(x)$ for $x$ not equal to a negative integer.
This functional equation has some resemblance with the
celebrated {\em duplication formula} for the
$\Gamma$ function:
$\Gamma(\frac{z}{2})\Gamma(\frac{z+1}{2}) = 2^{1-z} \sqrt{\pi}\, \Gamma(z)$.
 
\bigskip

We also point out the cancellation of $g(0)$ when we computed $g(1/2)$.
In particular,
we have not been able to give the value of $g(0)$ in terms of known 
constants.
The quantity $g(0) = \displaystyle\prod_{n \geq 1} \left(\tfrac{2n}{2n+1}\right)^{(-1)^{t_n}}$
already appeared in a paper of Flajolet and Martin \cite{FlaMar}: more precisely they are concerned with
the constant
$$
\varphi := 2^{-1/2}e^{\gamma}\frac{2}{3} R, \ \mbox{\rm where } 
R := \prod_{n \geq 1} \left(\frac{(4n+1)(4n+2)}{4n(4n+3)}\right)^{(-1)^{t_n}},$$
and it easily follows that
$$
\varphi = \frac{2^{-1/2} e^{\gamma}}{g(0)} \cdot 
$$

\bigskip

Finally, we will prove that the function $g$ is decreasing. Actually we have the stronger result
given in Theorem~\ref{decreasing} below. We first state and slighty extend a lemma 
(Lemmas~3.2 and 3.3 from \cite{AllCohen}).

\begin{lemma}\label{lemmas-all-cohen}
For every function $G$, define the operator ${\cal T}$ by ${\cal T}G(x) := G(2x)  - G(2x+1)$.
Let ${\mathcal A} = \{G : {\mathbb R}^+ \to {\mathbb R}, \  G \ \mbox{is } C^{\infty}, \
\forall x \geq 0, \ (-1)^r G^{(r)}(x) > 0\}$. Then

\begin{itemize}
\item for all $k \geq 0$,one has $T^k{\mathcal A} \subset {\mathcal A}$. Furthermore, if $G$ 
belongs to ${\mathcal A}$;

\item if the series $\sum_{n \geq 0} {\cal T}G(n)$ converges, then all the series 
$\sum_{n \geq 0} T^k G(n)$ converge and $R(n,T^kG) := \sum_{j \geq n} (-1)^{t_j} T^kG(j)$ 
has the sign of $(-1)^{t_n}$.
\end{itemize}
\end{lemma}

\begin{proof}
See \cite[Lemmas~3.2 and 3.2]{AllCohen} where everything is proved, except that the
last assertion about the sign of $R(n, T^k G)$ is stated only for $k=0$, but clearly holds
for all $k \geq 0$. 
\end{proof}

\begin{theorem}\label{decreasing}
The function $x \to f(\frac{x}{2}, \frac{x+1}{2})$ is decreasing on the nonnegative real numbers.
\end{theorem}

\begin{proof}
$G(x) = \log \frac{x + a}{x+b}$ for $x \geq 0$. 
Then, $G^{(r)}(x) = (-1)^{r-1}(r-1)!((x+a)^{-r} - (x+b)^{-r})$ for $r \geq 1$ so that $G$ belongs
to ${\mathcal A}$. Now applying Lemma~\ref{lemmas-all-cohen}  to ${\cal T}G$ and $n = 1$ yields
$\sum_{n \geq 1} (-1)^{t_n} (\log\tfrac{2n+a}{2n+b} - \log\tfrac{2n+1+a}{2n+1+b}) <  0$, which
is the same as saying that  $\tfrac{f(\frac{a}{2}, \frac{a+1}{2})}{f(\frac{b}{2}, \frac{b+1}{2})} < 1$.
\end{proof}

\section{Products of the form $\prod R(n)^{t_n}$}

We let again $(t_n)_{n \geq 0} = 0 \ 1 \ 1 \  0 \  1 \  0 \  0 \ 1 \  1 \  0 \  0 \  1 \ldots$ denote the 
0-1-Thue-Morse sequence. We have seen that several infinite products of the form 
$\prod R(n)^{(-1)^{t_n}}$ admit a closed-form expression, but it might seem more natural (or 
at least desirable) to have results for infinite products of the form $\prod R(n)^{t_n}$. Our first
result deals with the convergence of such products.

\begin{lemma}\label{num-den-0-1}
Let $t_n$ be the sum, reduced modulo $2$, of the binary digits of the integer $n$. Let 
$R \in {\mathbb C}(X)$ be a rational function such that the values $R(n)$ are defined for $n \geq 1$. 
Then the infinite product \, $\prod_n R(n)^{t_n}$ converges if and only if the numerator and 
the denominator of $R$ have the same degree, the
same leading coefficient, and the same sum of
roots (in ${\mathbb C}$).
\end{lemma}

\begin{proof}
If the infinite product $\prod R(n)^{t_n}$ converges, then $R(n)$ must tend to $1$ when 
$n$ tends to infinity (on the subsequence for which $t_n=1$). Hence the numerator and 
denominator of $R$ have the same degree and the same leading coefficient. But then, as 
we have seen, the product $\prod R(n)^{(-1)^{t_n}}$ converges, and so does the product 
$\prod R(n)^{2t_n + (-1)^{t_n}}$. But this product is equal to $\prod R(n)$,
which is known to
converge if and only if the sum of the roots of the numerator is equal to the 
sum of the roots of the 
denominator. 

Now if the numerator and denominator of $R$ have the same degree,
the same leading coefficient, 
and the same sum of roots, then both infinite products $\prod R(n)^{(-1)^{t_n}}$ and $\prod R(n)$
converge, which implies the convergence of the infinite product
$\prod R(n)^{1-2t_n} = \prod R(n)^{(-1)^{t_n}}$. 
\end{proof}

\bigskip

Now we give three equalities for products of the form $\prod R(n)^{t_n}$.

\begin{theorem}\label{three}
The following three equalities hold:
\begin{equation}
\prod_{n \geq 0} \left(\frac{(4n+1)(4n+4)}{(4n+2)(4n+3)}\right)^{t_n} = \frac{\pi^{3/4} \sqrt{2}}{\Gamma(1/4)} 
\end{equation}

\begin{equation}
\prod_{n \geq 0} \left(\frac{(n+1)(4n+5)}{(n+2)(4n+1)}\right)^{t_n} = \sqrt{2}
\end{equation}

\begin{equation}
\prod_{n \geq 0} \left(\frac{(8n+1)(8n+7)}{(8n+3)(8n+5)}\right)^{t_n}
=  \sqrt{2\sqrt{2} - 2}.
\end{equation}
\end{theorem}

\begin{proof}
As above we have $2t_n = 1 - (-1)^{t_n}$. Then we write
$$
\left(\prod_{n \geq 0} \left(\frac{(4n+1)(4n+4)}{(4n+2)(4n+3)}\right)^{t_n}\right)^2
= \frac
    {\displaystyle\prod_{n \geq 0} \frac{(4n+1)(4n+4)}{(4n+2)(4n+3)}}
    {\displaystyle\prod_{n \geq 0} \left(\frac{(4n+1)(4n+4)}{(4n+2)(4n+3)}\right)^{(-1)^{t_n}}}  \cdot
$$
The computation of the numerator is classical (see, e.g., \cite[Section~12-13]{ww}):
\begin{align*}
\displaystyle\prod_{n \geq 0} \frac{(4n+1)(4n+4)}{(4n+2)(4n+3)} &=
\displaystyle\prod_{n \geq 0} \frac{(n+1/4)(n+1)}{(n+1/2)(n+3/4)} = 
\frac{\Gamma(1/2)\Gamma(3/4)}{\Gamma(1/4)\Gamma(1)} =
\frac{\sqrt{\pi} \Gamma(3/4)}{\Gamma(1/4)} \\
&= \displaystyle\frac{\pi^{3/2} \sqrt{2}}{\Gamma(1/4)^2},
\end{align*}
where the last equality uses the reflection formula $\Gamma(x)\Gamma(1-x) = \pi/\sin(\pi x)$
for $x \notin {\mathbb Z}$.

\medskip

To compute the denominator, we start from the Woods-Robbins product and split the set
of indices into even and odd indices, so that
$$
\frac{\sqrt{2}}{2} = \prod_{n \geq 0}  \left(\frac{2n+1}{2n+2}\right)^{(-1)^{t_n}}
=  \prod_{n \geq 0}  \left(\frac{4n+1}{4n+2}\right)^{(-1)^{t_{2n}}}
     \left(\frac{4n+3}{4n+4}\right)^{(-1)^{t_{2n+1}}}.
$$
Using that $t_{2n} = t_n$ and $t_{2n+1} = 1 - t_n$, we thus have
$$
\frac{\sqrt{2}}{2} = \prod_{n \geq 0}  \left(\frac{(4n+1)(4n+4)}{(4n+2)(4n+3)}\right)^{(-1)^{t_n}}.
$$
Gathering the results for the numerator and for the denominator we deduce
$$
\left(\prod_{n \geq 0} \left(\frac{(4n+1)(4n+4)}{(4n+2)(4n+3)}\right)^{t_n}\right)^2
= \frac{2\pi^{3/2}}{\Gamma(1/4)^2},
$$
hence the first assertion in our theorem.

\bigskip

The proof of the second assertion goes along the same lines. We start from Equality~(h) in
Corollary~\ref{examples}
$$
\prod_{n \geq 0} \left(\frac{(n+1)(4n+5)}{(n+2)(4n+1)}\right)^{(-1)^{t_n}} = 2.
$$
Now
$$
\prod_{n \geq 0}  \left(\frac{(n+1)(4n+5)}{(n+2)(4n+1)}\right) = 
\prod_{n \geq 0}  \left(\frac{(n+1)(n+5/4)}{(n+2)(n+1/4)}\right) =
\frac{\Gamma(2)\Gamma(1/4)}{\Gamma(1)\Gamma(5/4)} = 4 .
$$
Note that this equality can also be obtained by telescopic cancellation in
the finite product  $\prod_{0 \leq n \leq N}  \left(\frac{(n+1)(4n+5)}{(n+2)(4n+1)}\right)$.

\medskip

Thus
$$
\prod_{n \geq 0} \left(\left(\frac{(n+1)(4n+5)}{(n+2)(4n+1)}\right)^{t_n}\right)^2 = 
\prod_{n \geq 0} \left(\frac{(n+1)(4n+5)}{(n+2)(4n+1)}\right)^{1-(-1)^{t_n}} = 2;
$$
hence
$$
\prod_{n \geq 0} \left(\frac{(n+1)(4n+5)}{(n+2)(4n+1)}\right)^{t_n} = \sqrt{2}.
$$

\bigskip

The proof of the third assertion is similar. We start from Equality~(l) in Corollary~\ref{examples}:
$$
\prod_{n \geq 0} \left(\frac{(8n+1)(8n+7)}{(8n+3)(8n+5)}\right)^{(-1)^{t_n}} = \frac{1}{2}  \cdot
$$
Now, as previously,
$$
\prod_{n \geq 0} \left(\frac{(8n+1)(8n+7)}{(8n+3)(8n+5)}\right) = 
\prod_{n \geq 0} \left(\frac{(n+1/8)(n+7/8)}{(n+3/8)(n+5/8)}\right) =
\frac{\Gamma(3/8)\Gamma(5/8)}{\Gamma(1/8)\Gamma(7/8)} \cdot
$$
But
$$
\frac{\Gamma(3/8)\Gamma(5/8)}{\Gamma(1/8)\Gamma(7/8)} =
\frac{\frac{\pi}{\sin(3\pi/8)}}{\frac{\pi}{\sin(\pi/8)}} = \frac{\sin(\pi/8)}{\sin(3\pi/8)}
= \frac{\sin(\pi/8)}{\cos(\pi/8)} = \tan(\pi/8) =  \sqrt{2} - 1 .
$$
From this we obtain
$$
\prod_{n \geq 0} \left(\frac{(8n+1)(8n+7)}{(8n+3)(8n+5)}\right)^{1 - (-1)^{t_n}}
=  2\sqrt{2} - 2.
$$
Thus, as claimed in the second assertion of the theorem,
$$ 
\prod_{n \geq 0} \left(\frac{(8n+1)(8n+7)}{(8n+3)(8n+5)}\right)^{t_n}
=  \sqrt{2\sqrt{2} - 2}.
$$
\end{proof}

\begin{remark}
Several other closed-form expressions for infinite products $\prod R(n)^{t_n}$ can be obtained.
For example one can use closed-form expressions for infinite products $\prod R(n)^{(-1)^{t_n}}$
where $R(n)$ satisfies the hypotheses of Lemma~\ref{num-den-0-1} and the classical result
about $\prod R(n)$. Another possibility is to start from an already known product 
$A = \prod_{n \geq 0} S(n)^{(-1)^{t_n}}$ where $S$ satisfies the hypotheses of 
Lemma~\ref{num-den} and note that (splitting the indexes into even and odd)
$$
A = \prod_{n \geq 0} (S(2n)^{(-1)^{t_{2n}}} S(2n+1)^{(-1)^{t_{2n+1}}})
= \prod_{n \geq 0} \left(\frac{S(2n)}{S(2n+1)}\right)^{(-1)^{t_n}}.
$$
As easily checked the rational function satisfies the hypotheses of Lemma~\ref{num-den-0-1}.
The reader can observe that this generalizes the method used to prove the first assertion (i.e.,
Equality~(1)) of Theorem~\ref{three}.
\end{remark}

\section{Generalization to another block counting sequence}\label{GS}

In this section we give an example of two products of the kind of those in Corollary~\ref{examples}
and in Theorem~\ref{three} that involve the Golay-Shapiro sequence (also called the 
Rudin-Shapiro sequence). Let us recall that this sequence $(v_n)_{n \geq 0}$ (in its binary version) 
can be defined as follows: $v_0 = 0$, and for all $n \geq 0$, $v_{2n} = v_n$, $v_{4n+1} = v_n$,
$v_{4n+3} = 1 - v_{2n+1}$.  Another definition is that
$v_n$ is the number, reduced modulo $2$,
of (possibly overlapping) $11$'s in the binary expansion of the integer $n$. The $\pm$ version
$((-1)^{v_n})_{n \geq 0}$ of this sequence was introduced independently the same year (1951) 
by Shapiro \cite{Shapiro} and by Golay \cite{Golay}, and rediscovered in 1959 by Rudin \cite{Rudin}
who acknowledged Shapiro's priority.

\medskip

An infinite product involving the Golay-Shapiro sequence was given in \cite{ACMFS} (also see
\cite{AS-infinite-products}):
$$
\prod_{n \geq 1} \left(\frac{(2n+1)^2}{(n+1)(4n+1)}\right)^{(-1)^{v_n}} = \frac{\sqrt{2}}{2}  \cdot 
$$
Here we prove the following theorem.

\begin{theorem}\label{gol-sha}
The following two equalities hold.
\begin{equation}
\prod_{n \geq 0} \left(\frac{4(n+2)(2n+1)^3(2n+3)^3}{(n+3)(n+1)^2(4n+3)^4}\right)^{(-1)^{v_n}} = 1
\end{equation}
\begin{equation}
\prod_{n \geq 0} \left(\frac{4(n+2)(2n+1)^3(2n+3)^3}{(n+3)(n+1)^2(4n+3)^4}\right)^{v_n} = 
\frac{16 \Gamma(3/4)^4}{\pi^6}  \cdot
\end{equation}
\end{theorem}

\begin{proof}
Let $R$ be a rational function in ${\mathbb C}(X)$ such that its numerator and denominator
have same degree and same leading coefficient. Suppose furthermore that $R(n)$ is defined
for any integer $n \geq 1$. Then it is not difficult to see that the infinite product 
$\prod_{n \geq 1} R(n)^{(-1)^{v_n}}$ converges (summation by part, given the well-known
property that the partial sum $\sum_{1 \leq k \leq n} (-1)^{v_k}$ is ${\mathcal O}(\sqrt{n})$).
Now, using the recursive definition of $(v_n)_{n \geq 0}$ one has
\begin{align*}
\prod_{n \geq 1} R(n)^{(-1)^{v_n}} &=
\prod_{n \geq 1} R(2n)^{(-1)^{v_{2n}}} \prod_{n \geq 0} R(2n+1)^{(-1)^{v_{2n+1}}}
\\
&= \prod_{n \geq 1} R(2n)^{(-1)^{v_n}}
\prod_{n \geq 0} R(4n+1)^{(-1)^{v_{4n+1}}} \prod_{n \geq 0} R(4n+3)^{(-1)^{v_{4n+3}}}
\\
&= 
\prod_{n \geq 1} R(2n)^{(-1)^{v_n}}
\prod_{n \geq 0} R(4n+1)^{(-1)^{v_n}} \prod_{n \geq 0} R(4n+3)^{-(-1)^{v_{2n+1}}} .
\end{align*}
Thus
$$
\prod_{n \geq 1} \left(\frac{R(n)}{R(2n) R(4n+1)}\right)^{(-1)^{v_n}}
\prod_{n \geq 0} R(4n+3)^{(-1)^{v_{2n+1}}}  = R(1)^{(-1)^{v_1}} = R(1).
$$
But
$$
\prod_{n \geq 0} R(4n+3)^{(-1)^{v_{2n+1}}}  = 
\frac{\displaystyle\prod_{n \geq 0} R(2n+1)^{(-1)^{v_n}}}
{\displaystyle\prod_{n \geq 0} R(4n+1)^{(-1)^{v_{2n}}}} =
\frac{\displaystyle\prod_{n \geq 1} R(2n+1)^{(-1)^{v_n}}}
{\displaystyle\prod_{n \geq 1} R(4n+1)^{(-1)^{v_n}}}
$$
which finally gives
$$
\prod_{n \geq 1} \left(\frac{R(n)R(2n+1)}{R(2n) R(4n+1)^2}\right)^{(-1)^{v_n}} =  R(1).
$$
Now taking $R(X) = \frac{(X+2)^2}{(X+1)(X+3)}$ we obtain
$$
\prod_{n \geq 1}
\left( \frac{4(n+2)(2n+1)^3(2n+3)^3}{(n+3)(n+1)^2(4n+3)^4}\right)^{(-1)^{v_n}} =  \frac{9}{8}\cdot
$$
Thus
$$
\prod_{n \geq 0} \left(\frac{4(n+2)(2n+1)^3(2n+3)^3}{(n+3)(n+1)^2(4n+3)^4}\right)^{(-1)^{v_n}} = 1.
$$
Now
\begin{align*}
\prod_{n \geq 0} \frac{4(n+2)(2n+1)^3(2n+3)^3}{(n+3)(n+1)^2(4n+3)^4}
&= 
\prod_{n \geq 0} \frac{(n+2)(n + \frac{1}{2})^3(n+\frac{3}{2})^3}{(n+3)(n+1)^2(n+\frac{3}{4})^4} \\
&= 
\frac{\Gamma(3) \Gamma(1)^2 \Gamma(3/4)^4}{\Gamma(2) \Gamma(1/2)^3 \Gamma(3/2)^3}
= \frac{16 \Gamma(3/4)^4}{\pi^6}\cdot
\end{align*}
\end{proof}

\section{Some questions}

The arithmetical properties values of the infinite products that we have obtained can be quite
different. Some are rational (e.g., $\prod_{n \geq 0}\left(\frac{(4n+1)(8n+7)}{(4n+2)(8n+3)}
\right)^{(-1)^{t_n}} = 1$ in Corollary~\ref{examples}~(k)), some are algebraic irrational such as 
the Woods-Robbins product ($\prod_{n \geq 0} \left(\frac{2n+1}{2n+2}\right)^{(-1)^{t_n}} =
\frac{\sqrt{2}}{2}$), some are transcendental 
(e.g., $\prod_{n \geq 0} \left(\frac{(4n+1)(4n+4)}{(4n+2)(4n+3)}\right)^{t_n} =
\frac{\pi^{3/4} \sqrt{2}}{\Gamma(1/4)}$ whose transcendency is a consequence of the algebraic 
independence of $\pi$ and $\Gamma(1/4)$ proved by \v{C}udnovs'ki\u{\i}; see \cite{Cudnovskii}).
As in \cite{All} some of these values could be proved transcendental if one admits the Rohrlich 
conjecture. A still totally open question is the arithmetical nature of the Flajolet-Martin constant(s)
(see the beginning of Section~\ref{FM}), namely $\varphi$ and $R$, where
$\varphi := 2^{-1/2}e^{\gamma}\frac{2}{3} R$, and 
$R := \prod_{n \geq 1} \left(\frac{(4n+1)(4n+2)}{4n(4n+3)}\right)^{(-1)^{t_n}} = \frac{3}{2g(0)} \cdot  $

\bigskip

Another question is to generalize the results for the Thue-Morse sequence to other sequences
counting certain patterns in the base-$b$ expansion of integers: the example given in 
Section~\ref{GS} is a first step in this direction. 

\section*{Acknowledgment}

This paper is an extended version of \cite{Riasat}. While we were preparing this extended
version, we found the paper \cite{vignat-wakhare} which has interesting results on finite (and 
infinite) sums involving the sum of digits of integers in integer bases.

\end{document}